\documentclass[11pt]{article}
\usepackage[top=20truemm,bottom=25truemm,left=15truemm,right=15truemm]{geometry}

\usepackage{amsmath,amssymb}
\usepackage{mathrsfs}
\usepackage{amsthm}
\usepackage{graphicx,bm}

  \makeatletter
    
    \@addtoreset{equation}{section}
  \makeatother 
\def\rn#1{\expandafter{\romannumeral #1}} 
\def\Rn#1{\uppercase\expandafter{\romannumeral #1}} 

\newcommand{\fr}{\frac}
\newcommand{\qq}{\qquad}
\newcommand{\la}{\lambda}

\newcommand{\Ga}{\Gamma}
\newcommand{\q}{\quad}
\newcommand{\e}{\varepsilon}

\newcommand{\sta}{\stackrel}

\newcommand{\R}{\mathbb{R}}
\newcommand{\C}{\mathbb{C}}
\newcommand{\Z}{\mathbb{Z}}
\newcommand{\N}{\mathbb{N}}

\newcommand{\B}{\mathcal{B}}
\renewcommand{\H}{\mathcal{H}}

\renewcommand{\L}{\mathcal{L}}

\newcommand{\Ran}{\mathrm{Ran}\,}

\newcommand{\dd}{\mathrm{d}}
\newcommand{\ff}{\mathrm{f}}

\newcommand{\A}{\mathcal{A}}

\DeclareMathOperator*{\op}{\oplus}

\DeclareMathOperator*{\ot}{\otimes}

\newcommand{\ded}{\hfill \mbox{\vrule \kern-.4pt \vbox to 10pt{\hrule width 3pt \vfill \hrule}\kern-.4pt \vrule} \par}

\theoremstyle{plain}

\newtheorem{theo}{Theorem}[section]

\newtheorem{lem}[theo]{Lemma}

\newtheorem{coro}[theo]{Corollary}

\newcommand{\F}{\mathcal{F}}
\newcommand{\supp}{\mathrm{supp}}
\newcommand{\La}{\Lambda}
\newcommand{\pp}{\mathrm{p}}

\newcommand{\ess}{\mathrm{ess}}
\newcommand{\sgn}{\mathrm{sgn}}
\newcommand{\cc}{\mathrm{c}}

\newcommand{\bfp}{\mathbf{p}}

\newcommand{\no}{\nonumber}

\begin{document}

\title{On the spectra of fermionic second quantization operators} 
\author{Shinichiro Futakuchi\footnote{Email: futakuchi@math.sci.hokudai.ac.jp} and Kouta Usui\footnote{Email: kouta@math.sci.hokudai.ac.jp}\\
\\
\textit{Department of Mathematics,} \\
\textit{Hokkaido University, Sapporo 060-0810, Japan.}}
\maketitle
\abstract{We derive several formulae for the spectra of the second quantization
operators in abstract fermionic Fock spaces.}

\section{Introduction}
Abstract theory of Fock spaces \cite{fock, BR,RS1,RS2} provides powerful mathematical tools when one analyzes models
of quantum field theory, the most promising physical 
theory which is expected to describe the fundamental
interactions of elementary particles. 
This results from the fact that 
quantum filed theory deals with a quantum system with infinitely many degrees of freedom,
including particles which may be created or annihilated, and that Fock spaces are furnished with suitable 
structure to describe particle creation or annihilation.
In mathematical physics, two different types of Fock spaces, bosonic (or symmetric) Fock space and fermionic
(or antisymmetric) Fock spaces, are considered, 
reflecting the fact that 
there are two different sorts of elementary particles in Nature --- bosons and fermions --- .

In mathematical analyses of quantum theories, one of the most important problems 
includes to determine the spectra of various self-adjoint operators representing
physical observables, especially, that of a Hamiltonian, which represents the total energy of the 
system under consideration. 
To each self-adjoint
operator $A$ acting in an underlying one particle Hilbert space $\H$, bosonic or fermionic 
second quantization is 
defined as an operator
which naturally ``lifts" $A$ up to the bosonic or fermionic Fock space over $\H$, respectively. 
In a bosonic Fock space,
the spectra of second quantization operators were well investigated and 
useful formulae for them have been available. However, as far as we know,
the corresponding useful formulae in a fermionic Fock space
are still missing. The main motivation of the present work is to derive such formulae in fermionic Fock spaces
to fill the gap.

Let $ \H $ be an infinite dimensional separable Hilbert space over $ \C $ with inner product $ \langle \cdot , \cdot  \rangle _\H  $ and norm $ \| \cdot  \| _\H $ (we omit the subscript $ \H $ if there will be no danger of confusion). For a linear operator $ T $ on $ \H $, we denote its domain by $ D(T) $. For a subspace $ D \subset D(T) $, the symbol $ T \upharpoonright D $ denotes the restriction of $ T $ to $ D $. We denote by $ \bar{T} $ the closure of $ T $ if $ T $ is closable. The spectrum (resp. the point spectrum) of $ T $ is denoted by $ \sigma (T) $ (resp. $ \sigma _\pp (T) $). The symbol $ \ot ^n \H $ (resp. $ \wedge ^n \H $) denotes the $ n $-fold tensor product of $ \H $ (resp. the $ n $-fold antisymmetric tensor product). Let $ \mathfrak{S} _n $ be the symmetric group of order $ n $. The antisymmetrization operator $ \A _n $ on $ \ot ^n \H $ is defined to be
\begin{align*}
\A _n := \fr{1}{n!} \sum _{\sigma \in \mathfrak{S} _n} \sgn (\sigma ) U_\sigma ,
\end{align*}
where $ U_\sigma $ is a unitary operator on $ \ot ^n \H $ such that $ U_\sigma (\psi _1 \ot \cdots \ot \psi _n) = \psi _{\sigma (1)} \ot \cdots \ot \psi _{\sigma (n)}  , \, \psi _j \in \H , \, j= 1, \dots , n $, and $ \sgn (\sigma ) $ is the signature of the permutation $ \sigma \in \mathfrak{S}_n $. Then, $ \A _n $ is an orthogonal projection onto $ \wedge ^n \H $. The fermionic Fock space over $ \H $ is defined by 
\begin{align*}
\F _\ff (\H ) := \op _{n=0} ^\infty \wedge ^n \H := \Big\{ \Psi = \{ \Psi ^{(n)} \} _{n=0} ^\infty \, \Big| \, \Psi ^{(n)} \in \wedge ^n \H , \, \sum _{n=0} ^\infty \| \Psi ^{(n)}  \| ^2 < \infty \Big\} .
\end{align*}
For a densely defined closable operator $ A $ on $ \H $ and $ j= 1, \dots , n $, we define a linear operator $ \widetilde{A}_j  $ on $ \ot ^n \H $ by
\begin{align*}
\widetilde{A}_j := I \ot \cdots \ot I \ot \sta{j\text{-th}}{\sta{\smallsmile}{A}} \ot I \ot \cdots \ot I ,
\end{align*}
where $ I $ denotes the identity. For each $ n \in \{ 0 \} \cup \N $, a linear operator $ A^{(n)} $ on $ \ot ^n \H $ is defined by 
\begin{align*}
A^{(0)} := 0 , \q A^{(n)} := \overline{  \sum _{j=1} ^n \widetilde{A}_j \upharpoonright \ot ^n _{\mathrm{alg}} D(A) } , \q n\ge 1 ,
\end{align*}
where $ \ot ^n _{\mathrm{alg}} D(A) $ means the $ n $-fold algebraic tensor product of $ D(A) $. Denote the reduced part of $ A^{(n)} $ (resp. $ \ot ^n A $) to $ \wedge ^n \H $ by $ A^{(n)} _\ff $ (resp. $ \wedge ^n A $). The infinite direct sum of these closed operators 
\begin{align*}
& \dd \Ga _\ff (A) := \op _{n=0} ^\infty A_\ff^{(n)} 
\end{align*}
is called the first type fermionic second quantization of $A$,
while the direct sum
\begin{align*}
& \Ga _\ff (A) := \op _{n=0} ^\infty \wedge ^n A
\end{align*}
is the second type.

\section{Main Results}
For a linear operator $ T $ on $\H$, $ \sigma _\mathrm{d} (T) $ denotes the discrete spectrum of $ T $. We introduce the notation
\begin{align}
t(\lambda;\{\lambda_1,\dots,\lambda_n\}):=\#\{j\,|\, \lambda=\lambda_j\},
\end{align}
or, in words, $t(\lambda;\{\lambda_1,\dots,\lambda_n\})$ represents how many 
$\lambda$'s appear in the set $\{\lambda_1,\dots,\lambda_n\}$.
The main results of the present paper are summarized in the following theorems:

\begin{theo}\label{main1} Let $ T $ be a self-adjoint operator on $\H$. 
Then, the following (i), (ii) hold.
\begin{enumerate}
\item The point spectrum of $T_\ff^{(n)}$ is given by 
\begin{align}\label{1.1.1}
\sigma _\pp (T_\ff ^{(n)})  = \Big\{ \sum _{j=1} ^n \la _j \, \Big| \, \la _j \in \sigma _\pp (T) \,(j=1,2,\dots,n), \, \,  t(\la _j;
\{\lambda_1,\dots,\lambda_n\}) \le \dim \ker (T-\la _j) \Big\}  .
\end{align}
\item If $ 0 \notin \sigma _\pp (T) $, then the 
point spectrum of $\wedge^n T$ is given by 
\begin{align}\label{1.1.2}
\sigma _\pp (\wedge ^n T) = \Big\{ \prod _{j=1} ^n \la _j \, \Big| \, \la _j \in \sigma _\pp (T) \, (j=1,2,\dots,n) , \,  
t(\la _j;\{\lambda_1,\dots,\lambda_n\}) \le \dim \ker (T-\la _j) \Big\} .
\end{align}
If $ 0 \in \sigma _\pp (T) $, then it is given by
\begin{align}\label{1.1.3}
\sigma _\pp (\wedge ^n T) = \{ 0 \} \cup \Big\{ \prod _{j=1} ^n \la _j \, \Big| \, \la _j \in \sigma _\pp (T) \, 
(j=1,2,\dots,n), \, 
 t(\la _j;\{\lambda_1,\dots,\lambda_n\}) \le \dim \ker (T-\la _j) \Big\} .
\end{align}
\end{enumerate}
\end{theo}

\begin{theo}\label{main2} Let $ T $ be a self-adjoint operator on $\H$. Then, the following (i), (ii) hold.
\begin{enumerate}
\item The spectrum of $T_\ff^{(n)}$ is given by 
\begin{align}\label{1.2.1}
\sigma (T^{(n)}_\ff) & = \overline{ \Big\{ \sum _{j=1} ^n \la _j \, \Big| \, \la _j \in \sigma (T) \,  (j=1,2,\dots,n), \, \, \text{if $ \la _j \in \sigma _\dd (T) $,} \, \,  t(\la _j;\{\lambda_1,\dots,\lambda_n\}) \le \dim \ker (T-\la _j) \Big\} } .
\end{align}
\item The spectrum of $\wedge ^n T$ is given by 
\begin{align}\label{1.2.2}
\sigma  (\wedge ^n T) = \overline{ \Big\{ \prod _{j=1} ^n \la _j \, \Big| \, \la _j \in \sigma (T) \,  (j=1,2,\dots,n), \, \, \text{if $ \la _j \in \sigma _\dd (T) $,} \, \,  t(\la _j;\{\lambda_1,\dots,\lambda_n\}) \le \dim \ker (T-\la _j) \Big\} } .
\end{align}
\end{enumerate}
\end{theo}

\section{Proof of the Theorems}

$ \H _\pp (T) $ denotes the closed linear subspace spanned by the eigenvectors of a linear operator $ T $ on $ \H $.

\begin{lem}\label{lem1.1} If $ T $ is a self-adjoint operator, then
\begin{enumerate}
\item 
\begin{align}\label{2.1}
\H _\pp (T^{(n)}) = \ot ^n \H _\pp (T) ,
\end{align}
\item 
\begin{align}\label{2.2}
\H _\pp (\ot ^n T) = \left[\ot ^n ( \H _\pp (T) \cap (\ker T) ^\perp ) \right]\op \ker (\otimes ^n T) .
\end{align}
\end{enumerate}
\end{lem}

\begin{proof} (i) Let $ \sigma _\mathrm{c} (T) := \sigma (T) \setminus \sigma _\mathrm{p} (T) $ and let $ \H _\cc (T) := \Ran E_T (\sigma _\cc (T)) $, where $ E_T (\cdot) $ is the one dimensional spectral measure of $ T $. Then, we obtain the direct sum decomposition $ T= T_\pp \oplus T_\cc  $ corresponding to the 
decomposition $ \H = \H _\pp (T) \op \H _\cc (T) $. Then, one can show that 
\begin{align}\label{directsum-decomposition}
T^{(n)} & = \op _{\sharp _j = \pp \, \, \text{or} \, \, \cc } \overline{ \sum _{j=1} ^n I_{\sharp_1} \otimes \cdots \otimes I_{\sharp _{j-1}} \otimes T_{\sharp_j} \otimes I_{\sharp_{j+1}} \otimes \cdots \otimes I_{\sharp_n} } \nonumber\\
& = \op _{\sharp _j = \pp \, \, \text{or} \, \, \cc } S_{\sharp _1 , \dots , \sharp _n} ,
\end{align}
where $ I_{\sharp _j} $ is the identity operator in $ \H _{\sharp _j} (T) $, and 
\[ S_{\sharp _1 , \dots , \sharp _n}:=\overline{ \sum _{j=1} ^n I_{\sharp_1} \otimes \cdots \otimes I_{\sharp _{j-1}} \otimes T_{\sharp_j} \otimes I_{\sharp_{j+1}} \otimes \cdots \otimes I_{\sharp_n} }. \]
Note that $ T_\cc $ have no eigenvalues. Since $ \sigma _\mathrm{p} (T) $ and $ \sigma _\mathrm{c} (T) $ are disjoint, one finds that,
 if $ \sharp _j =\cc $ for some $ j $, then $ \H _\pp (S_{\sharp _1 , \dots , \sharp _n}) = \{ 0 \} $ by using Lemma \ref{A.1}. Hence, in
 the direct sum decomposition in equation \eqref{directsum-decomposition}, only the term with $\sharp_j=p$, for all $j$, is nontrivial, which implies
\begin{align*}
\H _\pp (T^{(n)} ) & = \op _{\sharp _j = \pp \, \, \text{or} \, \, \cc } \H _\pp ( S_{\sharp _1 , \dots , \sharp _n} ) \\
& = \H _\pp ( S_{\pp , \dots , \pp }) .
\end{align*}
Since $ S_{\pp , \dots , \pp } $ is an operator in $ \ot ^n \H _\pp (T) $, it immediately 
follows that $ \H _\pp ( S_{\pp , \dots , \pp }) \subset \ot ^n \H _\pp (T) $. 

Conversely, let $\psi=\psi_1\otimes\dots\otimes\psi_n\in\otimes^n\H_\pp(T)$ with $\psi_j\in\ker(T-\lambda_j)$.
Then, direct computation shows
\begin{align}
T^{(n)}\psi&=\sum_{j=1}^n \psi_1\otimes \dots \otimes T\psi_j \otimes \dots\otimes \psi_n \no\\
&=\sum_{j=1}^n \psi_1\otimes \dots \otimes \lambda_j\psi_j \otimes \dots\otimes \psi_n \no\\
&=\sum_{j=1}^n \lambda_j \psi.
\end{align}
Thus, we see $\psi\in\ker(T^{(n)}-\lambda)$ with $\lambda=\sum_{j}\lambda_j$, and especially
$\psi\in \H_\pp(T^{(n)})$. Since the closed linear subspace spanned by such $\psi$'s is $\otimes^n \H_\pp(T)$
and $\H_\pp(T^{(n)})$ is closed, the converse inclusion follows.
This proves (\ref{2.1}).

(ii) Let $ T_0 $ and $ T_1 $ be the reduced parts of $ T $ by $ \ker T $ and $ (\ker T) ^\perp $ respectively. Then, we have a direct sum decomposition of $\H _\pp (\ot ^n T)$:
\begin{align*}
\H _\pp (\ot ^n T) = \op _{\flat _j = 0 \, \, \text{or} \, \, 1 } \H _\pp (T_{\flat _1} \ot \cdots \ot T_{\flat _n}) . 
\end{align*}
From this, we learn
\begin{align}\label{point-decomposition}
\H _\pp (\ot ^n T) = \H _\pp (\ot ^n T_1) \op \ker (\ot ^n T) ,
\end{align}
because, if $ \flat _j =0 $ for some $ j $, then $ T_{\flat _1} \ot \cdots \ot T_{\flat _n} $ is a null operator. 

Now, we will show, in general, that for self-adjoint operators $A,B$, whose point spectra do \textit{not} 
contain zero,
 \begin{align}\label{tensor-separating}
 \H_\pp(A\otimes B)=\H_\pp(A)\otimes \H_\pp(B).
 \end{align}  
In the same manner as in the proof of (i), we use the direct sum decompositions
$A=A_\pp \oplus A_\cc$ and $B=B_\pp \oplus B_\cc$ to obtain a decomposition of $\H_\pp(A\otimes B)$:
\begin{align}
\H_\pp(A\otimes B)=\H_\pp(A_\pp\otimes B_\pp)\oplus\H_\pp(A_\pp\otimes B_\cc)\oplus\H_\pp(A_\cc\otimes B_\pp)
\oplus\H_\pp(A_\cc\otimes B_\cc).
\end{align}
But, by Lemma \ref{A.1}, we have $\H_\pp(A_\cc\otimes B_\cc)=\{0\}$. 
Moreover, by the same Lemma, we also have $\H_\pp(A_\pp\otimes B_\cc)=\{0\}$ and 
$\H_\pp(A_\cc\otimes B_\pp)=\{0\}$, because we have assumed that $0\not\in\sigma_\pp(A)$
and $0\not\in\sigma_\pp(B)$. Therefore, one finds
\begin{align}
\H_\pp(A\otimes B)=\H_\pp(A_\pp\otimes B_\pp).
\end{align}
Since the operator $A_\pp\otimes B_\pp$ acts in $\H_\pp(A)\otimes\H_\pp (B)$, it is clear that
$\H_\pp(A_\pp\otimes B_\pp)\subset \H_\pp(A)\otimes\H_\pp (B)$, which means
\[  \H_\pp(A\otimes B)\subset \H_\pp(A)\otimes \H_\pp(B). \] 
The converse inclusion follows from the similar discussion given in (i). Thus, we prove \eqref{tensor-separating}.

The above general discussion shows that 
$ \H _\pp (\ot ^n T_1) = \ot ^n \H _\pp (T_1) $ since $T_1$ does not 
have zero eigenvalue. Substituting this equation in (\ref{point-decomposition}), we obtain (\ref{2.2}).
\end{proof}

\begin{proof}[Proof of Theorem \ref{main1}] 
(i) For vectors $ \zeta _1 , \dots , \zeta _n \in \H $, we define the wedge product of these vectors by
\begin{align*}
\zeta _1 \wedge \cdots \wedge \zeta _n := \sqrt{n!} \A _n (\zeta _1 \ot \cdots \ot \zeta _n) . 
\end{align*}
Let $ \{ \xi _k \} _k $ be a complete orthonormal system (CONS) of $ \H _\pp (T) $ consisting of eigenvectors of $ T $. Then, as is well known, the family
\begin{align}
\Lambda:=\{ \xi _{k_1} \wedge \cdots \wedge \xi _{k_n} \, \big| \, k_1 < \cdots < k_n \} 
\end{align}
forms a CONS of $ \wedge ^n \H _\pp (T) $, and each element is obviously an eigenvector of $ T^{(n)} $. By the reducibility and Lemma \ref{lem1.1} (i), we have
\begin{align}\label{wedge-comp}
\H _\pp (T_\ff ^{(n)}) = \A _n \H _\pp (T^{(n)}) = \A _n \ot ^n \H _\pp (T) = \wedge ^n \H _\pp (T) .
\end{align}

We claim 
\begin{align}\label{all-eigen}
\sigma_\pp(T_\ff^{(n)})=\left\{\lambda\in\R \,\Big|\, \text{There exists } \eta\in\Lambda \, \text{such that } T^{(n)}_\ff\eta=
\lambda\eta \right\}. 
\end{align}
The right hand side is clearly included by the left hand side. To prove
the converse, let $\lambda\in\sigma_\pp(T_\ff^{(n)})$
with an eigenvector $\psi$:
\[ T_\ff^{(n)}\psi=\lambda\psi. \]
Take the inner product with $\xi _{k_1} \wedge \cdots \wedge \xi _{k_n}$ to obtain
\begin{align}\label{cons-van}
\left(\sum_{j=1}^n \lambda_{k_j}-\lambda\right)\langle\xi _{k_1} \wedge \cdots \wedge \xi _{k_n}, \psi\rangle=0,
\end{align}
for all $k_1<\dots<k_n$, where $\lambda_{k_j}$ is an eigenvalue of $T$ to which $\xi_{k_j}$ belongs.
By equation \eqref{wedge-comp}, $\Lambda$ is a CONS of $\H_\pp(T^{(n)}_\ff)$, and thus,
for at least one choice of $(k_1,\dots,k_n)$, $\langle\xi _{k_1} \wedge \cdots \wedge \xi _{k_n}, \psi\rangle\not=0$.
This and equation \eqref{cons-van} imply 
\[ \lambda=\sum_{j=1}^n \lambda_{k_j}, \]
for such $(k_1,\dots,k_n)$. Since $\sum_{j=1}^n \lambda_{k_j}$ is an eigenvalue of $T^{(n)}_\ff$ to which 
 the eigenvector $\xi _{k_1} \wedge \cdots \wedge \xi _{k_n}$ belongs, $\lambda$ is
 an element of the right hand side of \eqref{all-eigen}.
 This proves \eqref{all-eigen}, but the right hand side of \eqref{all-eigen} is exactly the same set
 that appears in the right hand side of \eqref{1.1.1}.  Then, the proof is completed.

 (ii) In the case where $ 0 \notin \sigma _\pp (T) $, we can prove (\ref{1.1.2}) by using Lemma \ref{lem1.1} (ii) in the same way as in the proof of (\ref{1.1.1}). 

Next, suppose $ 0 \in \sigma _\pp (T) $. By the reducibility and Lemma \ref{lem1.1} (ii), we have
\begin{align}\label{decomp}
\H _\pp (\wedge ^n T) = (\wedge ^n \H _\pp (T_1)) \oplus \A _n \ker (\otimes ^n T) .
\end{align}
Let $ \{ \xi _k \} _k $ be a CONS of $\H_\pp(T_1)$ 
and
\[ \Lambda_1:=\{\xi_{k_1}\wedge\dots\wedge\xi_{k_n} \,|\, k_1<\dots<k_n\} ,\]
then $\Lambda_1$ forms a CONS of $\wedge^n \H_\pp(T_1)$.
We claim
\begin{align}\label{claim}
\sigma_\pp(\wedge^n T)\setminus\{0\}=\left\{\lambda\in\R\setminus\{0\} \,\Big|\, \text{There exists } \eta\in\Lambda_1 \, \text{such that } \wedge^n T\eta=\lambda\eta \right\}. 
\end{align}
Since the right hand side is clearly included by the left, it suffices to prove that, for each $\lambda\in\sigma_\pp(\wedge^n T)\setminus\{0\}$,
there is an $\eta\in\Lambda_1$ such that $\wedge^n T\eta=\lambda\eta$.
Let $\lambda\in \sigma_\pp(\wedge^n T)\setminus\{0\}$. Then, there exists a $ \psi \in \H _{\mathrm{p}} (\wedge ^n T) $ satisfying 
\[ \wedge^n T \psi = \lambda \psi . \]  
But since $ \la \neq 0 $, we may assume $ \psi \in \wedge ^n \H _\mathrm{p} (T_1) $ 
by \eqref{decomp}. Taking the inner product with $\xi_{k_1}\wedge\dots\wedge\xi_{k_n}$ on both sides,
one obtains
\begin{align}
 \left(\prod_{j=1}^n \lambda_{k_j}-\lambda\right)\langle\xi _{k_1} \wedge \cdots \wedge \xi _{k_n}, \psi\rangle=0,
 \end{align}
 for all $k_1<\dots<k_n$, where $\lambda_{k_j}$ is a non-zero eigenvalue of $T$ to which $\xi_{k_j}$ belongs.
 By noting that $\psi\in\wedge ^n \H _\pp (T_1)$, and $\Lambda_1$ is a CONS of $\wedge ^n \H _\pp (T_1)$,
 we learn that for at least one of $(k_1,\dots,k_n)$'s $\langle\xi _{k_1} \wedge \cdots \wedge \xi _{k_n}, \psi\rangle\not =0$.
 Thus, we have
 \[ \lambda=\prod_{j=1}^n\lambda_{k_j} \]
 for such $(k_1,\dots,k_n)$'s. But since $\prod_{j=1}^n\lambda_{k_j}$ is an eigenvalue to which the eigenvector
 $\xi_{k_1}\wedge\dots\wedge\xi_{k_n}\in\Lambda_1$ belongs, we proved the claim \eqref{claim}.

The right hand side of \eqref{claim} is rewritten as
\begin{align}
\Big\{ \prod _{j=1} ^n \la _j \, \Big| \, \la _j \in \sigma _\pp (T)\setminus\{0\} , \,  
t(\la _j;\{\lambda_1,\dots,\lambda_n\}) \le \dim \ker (T-\la _j) \Big\} ,
\end{align}
and therefore, we have proved
\begin{align}
\sigma_\pp(\wedge^n T)\setminus\{0\}=\Big\{ \prod _{j=1} ^n \la _j \, \Big| \, \la _j \in \sigma _\pp (T)\setminus\{0\} , \,  
t(\la _j;\{\lambda_1,\dots,\lambda_n\}) \le \dim \ker (T-\la _j) \Big\} ,
\end{align}
which implies \eqref{1.1.3}.
\end{proof}

\begin{proof}[Proof of Theorem \ref{main2}] (i) Let 
\[ \Sigma_n :=\Big\{ \sum _{j=1} ^n \la _j \, \Big| \, \la _j \in \sigma (T) , \, \, \text{if $ \la _j \in \sigma _\dd (T) $,} \, \,  t(\la _j;\{\lambda_1,\dots,\lambda_n\}) \le \dim \ker (T-\la _j) \Big\} \]
 be the right hand side of (\ref{1.2.1}) without closure.

First, we prove the left hand side includes the right. Let $ \la = \sum _{j=1} ^n \la _j , \, \la _j \in \sigma (T) $, and if $ \la _j \in \sigma _\dd (T) $, $ t( \la _j ;\{\lambda_1,\dots,\lambda_n\}) \le \dim \ker (T-\la _j ) $. Choose $\e_0>0$ in such a way that $\lambda_j\not=\lambda_k$
implies $U_{\e_0}(\lambda_j)\cap U_{\e_0}(\lambda_k)=\emptyset$, where $ U_\e (\la) $ is the $ \e $-neighborhood of $ \la$.
Then, for all $\e$ with $0<\e \le \e_0$, there exists an orthonormal set $ \{ \psi _j \} _{j=1} ^n \subset D (T) $ such that $ \psi _j \in \Ran E_T (U_\e (\la _j) ) , \,  \| (T- \la _j) \psi  \| <\e  $. Then, we obtain
\begin{align*}
& \| (T^{(n)} _\ff -\la ) \psi _1 \wedge \cdots \wedge \psi _n \| \\
\le & \| \sum _{j=1} ^n \psi _1 \wedge \cdots \wedge (T-\la _j) \psi _j \wedge \cdots \wedge \psi _n \| \\
\le & \sqrt{n!} \e . 
\end{align*}
Hence, $ \la \in \sigma (T_\ff ^{(n)} ) $.

Next, in order to prove the converse inclusion, we will show that there exists a CONS $ \{ \Psi _k \} _{k=1} ^\infty \subset \bigwedge ^n \H $ satisfying the following condition:
\begin{align}\label{supp}
\supp \, \mu _{\Psi _k} \subset \overline{\Sigma_n} , \qq  k=1,2,\dots ,
\end{align}
where $ \mu _{\Psi _k} (B) := \| E_{T_\ff ^{(n)}} (B) \Psi _k \| ^2 $ for a one dimensional borel set $ B \in \B ^1 $. Then, it immediately follows that $ \bigwedge ^n \H = \overline{\L } ( \{ \Psi _k \} _k )  \subset \Ran E_{T_\ff ^{(n)}} (\overline{\Sigma _n}) $. Here,
$\overline{\L } ( \{ \Psi _k \} _k )$ denotes the closed linear supspace spanned by $\{\Psi_k\}_k$. Hence, we obtain 
$E_{T_\ff^{(n)}}(\overline{\Sigma_n})=I $, and this implies $ \sigma (T_\ff ^{(n)} ) =\supp E_{T_\ff^{(n)}}\subset \overline{\Sigma _n } $ by definition of support of a spectral
measure.

Now, we shall show (\ref{supp}). Denote the essential spectrum of $ T $ by $ \sigma _\mathrm{ess} (T) $ and introduce the notations $ \H _\dd (T):=\Ran E_\ff^{(n)}(\sigma_\dd(E_\ff^{(n)})) $ and $\H _\ess (T):=\Ran E_\ff^{(n)}(\sigma_\ess(E_\ff^{(n)})) $. Let
$ \{ \zeta _k \} _k $ be a CONS of $ \H _\dd (T) $ consisting of eigenvectors of $ T $, and let $ \{ \eta _l \} _j $ be a CONS of $ \H _ \ess (T) $. Then, the set
\begin{align*}
\La := \{ \zeta _{k_1} \wedge \cdots \wedge \zeta _{k_N} \wedge \eta _{l_1} \wedge \cdots \wedge \eta _{l_{N'}} \, \Big| \, k_1 < \cdots < k_N , \, l_1 < \cdots < l_{N'} , \, N+N' =n \} 
\end{align*}
forms a CONS of $ \bigwedge ^n \H $. Fix some $ \Psi = \zeta _{k_1} \wedge \cdots \wedge \zeta _{k_N} \wedge \eta _{l_1} \wedge \cdots \wedge \eta _{l_{N'}} \in\Lambda$. In addition, let $ \nu _{k_i} (B) := \| E_T (B) \zeta _{k_i } \| ^2 \, (i=1, \dots , N) , \, \nu _{l_j } (B) := \| E_T (B) \eta _{l_j} \| \, (j=1, \dots , N') $ for $ B \in \B ^1 $, and $ B_i := \supp \, \nu  _{k_i} $, $ B_{N+j} := \supp \, \nu _{l_j} $. Let 
\begin{align*}
J &:= \bigcup _{\sigma \in \mathfrak{S}_n } B_{\sigma (1)} \times \cdots \times B_{\sigma (n)} ,\nonumber \\
J_{\Sigma_n}&:=\{ (\la _1 , \dots , \la _n) \in \R ^n \, | \, \la_1 + \cdots + \la _n \in \Sigma_n \},
\end{align*}
and $E_T^n(\cdot)$ be an $n$-dimensional spectral measure acting in $\otimes^n\H$ defined by the relation
\begin{align*}
E_T^n(B_1\times\dots\times B_n)=E_T(B_1)\otimes\dots\otimes E_T(B_n), \quad B_1,\dots,B_n\in\mathcal{B}^1.
\end{align*}
Then, we see that $ J \subset J_{\Sigma_n } $ by the construction of $ \La $. By direct computation, we have $ E^n _T (J) \Psi =\Psi $. Therefore, it follows that $E_T^n(J_{\Sigma _n}) \Psi=\Psi$ and one finds
\begin{align*}
E_{T_\ff ^{(n)}} (\Sigma_n ) \Psi & = E_{T^{(n)}} (\Sigma_n ) \A _n \Psi = E_T ^n (J_{\Sigma_n} ) \Psi \\
& = \Psi ,
\end{align*}
where Lemma \ref{A.2} is used to obtain the second equality.
This shows that $\supp\mu_{\Psi}\subset\overline{\Sigma_n}$, and therefore, $\Lambda$ is a desired CONS satisfying \eqref{supp}, completing the proof of (\ref{1.2.1}). 

(ii) The proof is very similar to that of (i), and we omit it. 
\end{proof}

As a corollary of these theorems, we can derive the formulae for spectra of the fermonic second quantization 
operators. In what follows, we will use the simpler notation $t(\lambda_j)$ in place of $t(\lambda_j;\{\lambda_1,\dots,\lambda_n\})$ for notational simplicity.

\begin{coro}\label{second-quantization} 
Let $ T $ be a self-adjoint operator. Then, the following (i), (ii) hold.
\begin{enumerate}
\item The spectrum and the point spectrum of the first type fermionic second quantization operator of $ T $ are given by
\begin{align*}
& \sigma (\dd \Ga _\ff  (T)) = \{ 0 \} \cup \overline{ \Big( \bigcup _{n=1} ^\infty \Big\{ \sum _{j=1} ^n \la _j \, \Big| \, \la _j \in \sigma (T) , \, j=1, \dots , n , \, \text{if $ \la _j \in \sigma _\dd (T) $}, \, t( \la _j ) \le \dim \ker (T-\la _j ) \Big\} \Big)  } , \\
& \sigma _\pp  (\dd \Ga _\ff  (T)) = \{ 0 \} \cup \Big( \bigcup _{n=1} ^\infty \Big\{ \sum _{j=1} ^n \la _j \, \Big| \, \la _j \in \sigma _\pp (T) , \, j=1, \dots , n , \,  t( \la _j ) \le \dim \ker (T-\la _j ) \Big\} \Big)  .
\end{align*}
\item As to the second type fermionic second quantization operator of $ T $, one has
\begin{align*}
\sigma (\Ga _\ff  (T)) = \{ 1 \} \cup \overline{ \Big( \bigcup _{n=1} ^\infty \Big\{ \prod _{j=1} ^n \la _j \, \Big| \, \la _j \in \sigma (T) , \, j=1, \dots , n , \,  \text{if $ \la _j \in \sigma _\dd (T) $}, \,t( \la _j ) \le \dim \ker (T-\la _j ) \Big\} \Big) }.
\end{align*}
If $ 0 \notin \sigma _\pp (T) $, then
\begin{align*}
\sigma _\pp ( \Ga _\ff  (T)) = \{ 1 \} \cup \Big( \bigcup _{n=1} ^\infty \Big\{ \sum _{j=1} ^n \la _j \, \Big| \, \la _j \in \sigma _\pp (T) , \, j=1, \dots , n , \,  t( \la _j ) \le \dim \ker (T-\la _j ) \Big\} \Big) .
\end{align*}
If $ 0 \in \sigma _\pp (T) $, then
\begin{align*}
\sigma _\pp ( \Ga _\ff  (T)) = \{ 0 \} \cup \{ 1 \} \cup \Big( \bigcup _{n=1} ^\infty \Big\{ \sum _{j=1} ^n \la _j \, \Big| \, \la _j \in \sigma _\pp (T) , \, j=1, \dots , n , \,  t( \la _j ) \le \dim \ker (T-\la _j ) \Big\} \Big) .
\end{align*}
\end{enumerate}
\end{coro}

\section{Example --- Kinetic energy of free fermions in a finite box ---}
Let $\F$ be the fermionic Fock space over $\H=l^2(\Gamma_L;\C^4)$, where
\[ \Gamma_L:=\frac{2\pi}{L}\Z^3,\quad L>0 .\]
 This Hilbert space $\H$ consists
of quantum mechanical state vectors of one Dirac fermion in momentum representation living in 
a finite volume box $[-L/2,L/2]^3\subset \R^3$. As a one particle Hamiltonian, we adopt a multiplication operator by a function $E_M$:
\[ E_M(\bfp)=\sqrt{\bfp^2+M^2}, \quad \bfp\in \Gamma_L, \]
where $\bfp\in\R^3$ is a spacial momentum of a Dirac particle and $M\ge0$ is a constant 
representing a bare mass of a Dirac particle. On a spinor space $\C^4$, $E_M$ acts as a diagonal 
matrix.

The spectrum of $E_M$ is given as follows:
\begin{lem}\label{spec-one-part}
\begin{enumerate}
\item The spectrum of $E_M$ is given by
\begin{align}
\sigma(E_M)=\sigma_\dd(E_M)=\left\{ \sqrt{\bfp^2+M^2}\,\Big|\,\bfp\in\Gamma_L \right\}.
\end{align}
\item The multiplicity of eigenvalue $ \la $ is given by
\begin{align}
\dim\ker(E_M-\lambda)=4r\left(\frac{L^2}{4\pi^2}(\lambda^2-M^2)\right),
\end{align}
where $r(N)$ denotes 
\[ r(N):=\#\{\mathbf{n}\in\Z^3\,|\, N=\mathbf{n}^3 \}.\]
\end{enumerate}
\end{lem}

\begin{proof}
Since (i) is well known, we will prove only (ii).

For each $(\bfp, l)\in\Gamma_L\times\{1,2,3,4\}$, let
\[ \delta_\bfp^l(\mathbf{q},m)=\delta_{\bfp\mathbf{q}}\delta_{lm}, \quad \mathbf{q}\in \Gamma_L, \,m=1,2,3,4 .\]
Then, $\{ \delta_\bfp^l\}_{\bfp,l}$ forms a CONS of $\H$ under natural identification $\H=l^2(\Gamma_L
\times\{1,2,3,4\})$.
From a general theory of multiplication operators, we have
\begin{align}
\ker(E_M -\lambda ) = \{ \psi\in D(E_M)\,|\, \psi(\bfp)\not=0 \text{ implies }\sqrt{\bfp^2+M^2}=\lambda\}.
\end{align} 
This means that $\psi$ is an eigenvector of $E_M$ if and only if it belongs to the linear subspace spanned by
\[ \{\delta_{\bfp}^l\in \H \,|\,\sqrt{\bfp^2+M^2}=\lambda,\,l=1,2,3,4\}. \]
Since the above vectors are linearly independent, we find
\begin{align} 
\dim\ker (E_M-\lambda)&=4\cdot \#\{\bfp \in\Gamma_L\,|\, \sqrt{\bfp^2+M^2}=\lambda \} \nonumber\\
&=4\cdot r\left(\frac{L^2}{4\pi^2}(\lambda^2-M^2)\right).
\end{align}
\end{proof}

From Lemma \ref {spec-one-part} and Corollary \ref{second-quantization}, we finally
 arrive at the formula for the second quantization operator $\dd \Ga _\ff  (E_M)$ acting in $\F$: 
\begin{theo}
\begin{align}
\sigma (\dd \Ga _\ff  (E_M)) 
=\{ 0 \} \cup \overline{ \Big( \bigcup _{n=1} ^\infty \Big\{ \sum _{j=1} ^n 
\sqrt{\frac{4\pi^2}{L^2}N_j+M^2} \, \Big| \, 1\le r(N_j),\,t(N_j;\{N_1,\dots,N_n\}) \le 4r(N_j), \;  j=1, \dots , n  \Big\} \Big)  } .
\end{align}
\end{theo}
The crucial difference from the bosonic case is the 
existence of the restriction $t(N_j;\{N_1,\dots,N_n\}) \le 4r(N_j)$ reflecting Pauli's exclusion 
principle. It would be interesting to note that, if Dirac fermions are not contained in a finite box
but live in $\R^3$,
then there is no restriction because the spectrum of $E_M$ in this case consists only of essential
spectra.

\section*{Acknowledgements}
The authors are grateful to Prof. Asao Arai for valuable comments. This work is partially
supported by JSPS Research Fellowships for Young Scientists and 
by World Premier International Center Initiative (WPI Program), MEXT, Japan.

\appendix
\section{Appendix} 
We will collect well known facts in abstract Fock spaces used in this paper.
Detailed proofs will be found in \cite{fock}.
\begin{lem}\label{A.1} Let $ A_j \, (j=1 , \cdots , n) $ be self-adjoint operators on
a separable Hilbert space $\H_i$. Then,

(i) 
\begin{align*}
\sigma _\pp \Big( \overline{ \sum _{j=1} ^n \widetilde{A}_j } \Big) = \Big\{ \sum _{j=1} ^n \la _j \, \Big| \, \la _j \in \sigma _\pp (A_j), \, j=1 , \dots , n \Big\} . 
\end{align*}
(ii) If $ 0 \notin \sigma _\pp (A_j) $ for all $ A_j $, then
\begin{align*}
\sigma _\pp \Big( \ot  _{j=1} ^n A_j \Big) = \Big\{ \prod _{j=1} ^n \la _j \, \Big| \, \la _j \in \sigma _\pp (A_j), \, j=1 , \dots , n \Big\} . 
\end{align*}
If $ 0 \in \sigma _\pp (A_j) $ for a $ A_j $, then
\begin{align*}
\sigma _\pp \Big( \ot  _{j=1} ^n A_j \Big) = \{ 0 \} \cup \Big\{ \prod _{j=1} ^n \la _j \, \Big| \, \la _j \in \sigma _\pp (A_j), \, j=1 , \dots , n \Big\} . 
\end{align*}
\end{lem}

\begin{lem}\label{A.2} Let $ A_j \, (j=1 , \cdots , n) $ be self-adjoint operators  on 
a Hilbert space $\H_i$. Let $ E := E_{\widetilde{A}_1} \times \cdots \times E_{\widetilde{A}_n} $ be the product measure and $\mathcal{B}_1$ be the Borel field of $\R$. Then,

(i) 
\begin{align*}
E_{\Sigma } (J) := E \Big( \Big\{ (\la _1 , \dots , \la _n ) \in \R ^n \, \Big| \, \sum _{j=1} ^n \la _j \in J \Big\} \Big)  , \qq J \in \B ^1 
\end{align*}
is the spectral measure of $ \overline{ \sum _{j=1} ^n \widetilde{A}_j } $.

(ii) 
\begin{align*}
E_{\otimes } (J) := E \Big( \Big\{ (\la _1 , \dots , \la _n ) \in \R ^n \, \Big| \, \prod _{j=1} ^n \la _j \in J \Big\} \Big) , \qq J \in \B ^1
\end{align*}
 is the spectral measure of $ \ot  _{j=1} ^n A_j $.
\end{lem}

\end{document}